\documentclass[12pt]{amsart}
\usepackage{amssymb, amsmath}
\title{A new characterization of $q_{\omega}$-compact algebras}
\author{M. Shahryari}
\address{M. Shahryari: Department of Pure Mathematics,  Faculty of Mathematical
Sciences, University of Tabriz, Tabriz, Iran}
\email{mshahryari@tabrizu.ac.ir}

\newcommand{\Rad}{\mathrm{Rad}}
\newcommand{\LL}{\mathcal{L}}
\newcommand{\VV}{\mathbf{V}}
\newcommand{\FX}{F_{\mathbf{V}}(n)}
\newcommand{\SigV}{\Sigma(\mathbf{V})}
\newcommand{\ConV}{\mathrm{Con}(\mathbf{V})}
\newcommand{\PconV}{\mathrm{PCon}(\mathbf{V})}
\newcommand{\qo}{q_{\omega}}

\newtheorem {theorem}{Theorem}
\newtheorem{lemma}{Lemma}

\newtheorem{definition}{Definition}
\begin{document}

\maketitle
\begin{abstract}
In this note, we give a new characterization for an algebra to be $\qo$-compact in terms of {\em super-product operations} on the lattice of congruences of the relative free algebra.
\end{abstract}


\section{Introduction}
In this article, our notations are the same as \cite{DMR1}, \cite{DMR2}, \cite{DMR3}, \cite{DMR4} and \cite{ModSH}. The reader should review these references for a complete account of the universal algebraic geometry. However, a brief review of fundamental notions will be given in the next section.

Let $\LL$ be an algebraic language,   $A$ be an algebra of type $\LL$ and  $S$ be a system of equation in the language $\LL$. Recall that an equation $p\approx q$ is a logical consequence of $S$ with respect to $A$, if any solution of $S$ in $A$ is also a solution of $p\approx q$. The radical $\Rad_A(S)$ is the set of all logical consequences of  $S$ with respect to  $A$. This radical is clearly a congruence of the term algebra $T_{\LL}(X)$ and in fact it is the largest subset of the term algebra which is equivalent to $S$ with respect to $A$. Generally, this logical system of equations with respect to $A$ does not obey the ordinary compactness of the first order logic. We say that an algebra $A$ is  $\qo$-compact, if for any system $S$ and any consequence $p\approx q$, there exists a finite subset $S_0\subseteq S$ with the property that $p\approx q$ is a consequence of $S_0$ with respect to $A$. This property of being $\qo$-compact is equivalent to
$$
\Rad_A(S)=\bigcup_{S_0}\Rad_A(S_0),
$$
where $S_0$ varies  in the set of all finite subsets of $S$. If we look at the map $\Rad_A$ as a closure operator on the lattice of systems of equations in the language $\LL$, then we see that $A$ is $\qo$-compact if and only if $\Rad_A$ is an algebraic. The class of $\qo$-compact algebras is very important and it contains many elements. For example, all equationally noetherian algebras belong to this class. In \cite{DMR3}, some equivalent conditions for $\qo$-compactness are given. Another equivalent condition is obtained in \cite{MR} in terms of {\em geometric equivalence}. It is proved that (the proof is implicit in \cite{MR}) an algebra $A$ is $\qo$-compact if and only if $A$ is geometrically equivalent to any of its filter-powers. We will discuss geometric equivalence in the next section. We will use this fact of \cite{MR} to obtain a new characterization of $\qo$-compact algebras. Although our main result will be formulated in an arbitrary variety of algebras, in this introduction, we give a simple description of this result for the case of the variety of all algebras of type $\LL$.

Roughly speaking, a {\em super-product operation} is a map $C$ which takes a set $K$ of congruences of the term algebra and returns a new congruence $C(K)$ such that for all $\theta\in K$, we have $\theta\subseteq C(K)$. For an algebra $B$ define a map $T_B$ which takes a system $S$ of equations and returns
$$
T_B(S)=\{ \Rad_B(S_0):\ S_0\subseteq S,\ |S_0|<\infty\}.
$$
Suppose for all algebra $B$ we have $C\circ T_B\leq \Rad_B$. We prove that an algebra $A$ is $\qo$-compact if and only if $C\circ T_A=\Rad_A$.

\section{Main result}
Suppose $\LL$ is an algebraic language. All algebras we are dealing with, are of type $\LL$. Let $\VV$ be a variety of algebras. For any $n\geq 1$, we denote the relative free algebra of $\VV$, generated by the finite set $X=\{ x_1, \ldots, x_n\}$, by $\FX$. Clearly, we can assume that an arbitrary element $(p, q)\in \FX^2$ is an equation in the variety $\VV$ and we can denote it by $p\approx q$. We introduce the following list of notations:\\

1- $\mathrm{P}(\FX^2)$ is the set of all systems of equations in the variety $\VV$.\\

2- $\mathrm{Con}(\FX)$ is the set of all congruences of $\FX$.\\

3- $\Sigma(\VV)=\bigcup_{n=1}^{\infty} P(\FX^2)$.\\

4- $\ConV= \bigcup_{n=1}^{\infty}\mathrm{Con}(\FX)$.\\

5- $\PconV= \bigcup_{n=1}^{\infty}\mathrm{P}(\mathrm{Con}(\FX))$.\\

6- $\qo(\VV)$ is the set of all $\qo$-compact elements of $\VV$.\\
\newline
Note that, we have $\ConV\subseteq \SigV$. For any algebra $B\in \VV$, the map $\Rad_B:\SigV\to \SigV$ is a closure operator and $B$ is $\qo$-compact, if and only if this operator is algebraic. Define a map
$$
T_B:\SigV\to \PconV
$$
by
$$
T_B(S)=\{ \Rad_B(S_0):\ S_0\subseteq S,\ |S_0|<\infty\}.
$$

\begin{definition}
A map $C:\PconV\to \ConV$ is  called a super-product operation, if for any $K\in \PconV$ and $\theta\in K$, we have $\theta\subseteq C(K)$.
\end{definition}

There are many examples of such operations; the ordinary product of normal subgroups in the varieties of groups is the simplest one. For another example, we can look at the map $C(K)=\Rad_B(\bigcup_{\theta\in K}\theta)$, for a given fixed $B\in \VV$. We are now ready to present our main result.

\begin{theorem}
Let $C$ be a super-product operation such that for any $B\in\VV$, we have $C\circ T_B\leq \Rad_B$. Then
$$
\qo(\VV)=\{ A\in \VV:\ C\circ T_A=\Rad_A\}.
$$
\end{theorem}

To prove the theorem, we first give a proof for the following claim. Note that it is implicitly proved in \cite{MR} for the case of groups. \\

{\em An algebra is $\qo$-compact if and only if it is geometrically equivalent to any of its filter-powers.}\\

Let $A\in \VV$ be a $\qo$-compact algebra and $I$ be a set of indices. Let $F\subseteq P(I)$ be a filter and $B=A^I/F$ be the corresponding filter-power. We know that the quasi-varieties generated by $A$ and $B$ are the same. So, these algebras have the same sets of quasi-identities. Now, suppose that $S_0$ is a finite system of equations and $p\approx q$ is another equation. Consider the following quasi-identity
$$
\forall \overline{x} (S_0(\overline{x})\to p(\overline{x})\approx q(\overline{x})).
$$
This quasi-identity is true in $A$, if and only if it is true in $B$. This shows that $\Rad_A(S_0)=\Rad_B(S_0)$.
Now, for an arbitrary system $S$, we have
\begin{eqnarray*}
\Rad_A(S)&=&\bigcup_{S_0}\Rad_A(S_0)\\
         &=&\bigcup_{S_0}\Rad_B(S_0)\\
         &\subseteq&\Rad_B(S).
\end{eqnarray*}
Note that in the above equalities, $S_0$ ranges in the set of finite subsets of $S$. Clearly, we have $\Rad_B(S)\subseteq \Rad_A(S)$, since $A\leq B$. This shows that $A$ and $B$ are geometrically equivalent. To prove the converse, we need to define some notions. Let $\mathfrak{X}$ be a prevariety, i.e. a class of algebras closed under product and subalgebra. For any $n\geq 1$, let $F_{\mathfrak{X}}(n)$ be the free element of $\mathfrak{X}$ generated by $n$ elements. Note that if $\VV=var(\mathfrak{X})$, then $F_{\mathfrak{X}}(n)=\FX$. A congruence $R$ in $F_{\mathfrak{X}}(n)$ is called an $\mathfrak{X}$-radical, if $F_{\mathfrak{X}}(n)/R\in \mathfrak{X}$. For any $S\subseteq F_{\mathfrak{X}}(n)^2$, the least $\mathfrak{X}$-radical containing $S$ is denoted by $\Rad_{\mathfrak{X}}(S)$.

\begin{lemma}
For an algebra $A$ and any system $S$, we have
$$
\Rad_A(S)=\Rad_{pvar(A)}(S),
$$
where $pvar(A)$ is the prevariety generated by $A$.
\end{lemma}

\begin{proof}
Since $F_{\mathfrak{X}}(n)/\Rad_A(S)$ is a coordinate algebra over $A$, so it embeds in a direct power of $A$ and hence it is an element of $pvar(A)$. This shows that
$$
\Rad_{pvar(A)}(S)\subseteq \Rad_A(S).
$$
Now, suppose $(p,q)$ does not belong to $\Rad_{pvar(A)}(S)$. So, there exists $B\in pvar(A)$ and a homomorphism $\varphi:F_{\mathfrak{X}}(n)\to B$ such that $S\subseteq \ker \varphi$ and $\varphi(p)\neq \varphi(q)$. But, $B$ is separated by $A$, hence there is a homomorphism $\psi:B\to A$ such that
$\psi(\varphi(p))\neq \psi(\varphi(q))$. This shows that $(p, q)$ does not belong to $\ker (\psi\circ \varphi)$. Therefore, it is not in $\Rad_A(S)$.
\end{proof}

Note that, since $pvar(A)$ is not axiomatizable in general, so we can not give a deductive description of elements of $\Rad_A(S)$. But, for $\Rad_{var(A)}(S)$ and $\Rad_{qvar(A)}(S)$ this is possible, because the variety and quasi-variety generated by $A$ are axiomatizable. More precisely, we have:\\

1- Let $\mathrm{Id}(A)$ be the set of all identities of $A$. Then $\Rad_{var(A)}(S)$ is the set of all logical consequences of $S$ and $\mathrm{Id}(A)$.\\

2- Let  $\mathrm{Q}(A)$ be the set of all identities of $A$. Then $\Rad_{qvar(A)}(S)$ is the set of all logical consequences of $S$ and $\mathrm{Q}(A)$.\\

We can now, prove the converse of the claim. Suppose $A$ is not $\qo$-compact. We show that
$$
pvar(A)_{\omega}\neq qvar(A)_{\omega}.
$$
Recall that for and arbitrary class $\mathfrak{X}$, the notation $\mathfrak{X}_{\omega}$ denotes the class of finitely generated elements of $\mathfrak{X}$. Suppose in contrary we have the equality
$$
pvar(A)_{\omega}= qvar(A)_{\omega}.
$$
Assume that $S$ is an arbitrary system and $(p, q)\in \Rad_A(S)$. Hence, the infinite quasi-identity
$$
\forall \overline{x}(S(\overline{x})\to p(\overline{x})\approx q(\overline{x}))
$$
is true in $A$. So, it is also true in $pvar(A)$. As a result, every element from $qvar(A)_{\omega}$ satisfies this infinite quasi-identity. Let $F_A(n)=F_{var(A)}(n)$. We have $F_A(n)\in qvar(A)_{\omega}$ and hence $\Rad_{qvar(A)}(S)$ depends only on $qvar(A)_{\omega}$. In other words, $(p, q)\in \Rad_{qvar(A)}(S)$, so $p\approx q$ is a logical consequence of the set of $S+\mathrm{Q}(A)$. By the compactness theorem of the first order logic, there exists a finite subset $S_0\subseteq S$ such that $p\approx q$ is a logical consequence of $S_0+\mathrm{Q}(A)$. This shows that $(p, q)\in\Rad_{qvar(A)}(S_0)$. But $\Rad_{qvar(A)}(S_0)\subseteq \Rad_A(S_0)$. Hence $(p, q)\in \Rad_A(S_0)$, violating our assumption of non-$\qo$-compactness of $A$. We now showed that
$$
pvar(A)_{\omega}\neq qvar(A)_{\omega}.
$$
By the algebraic characterizations of the classes $pvar(A)$ and $qvar(A)$, we have
$$
SP(A)_{\omega}\neq SPP_u(A)_{\omega},
$$
where $P_u$ is the ultra-product operation. This shows that there is an ultra-power $B$ of $A$ such that
$$
SP(A)_{\omega}\neq SP(B)_{\omega}.
$$
In other words the classes  $pvar(A)_{\omega}$ and $pvar(B)_{\omega}$ are different. We claim that $A$ and $B$ are not geometrically equivalent. Suppose this is not the case. Let $A_1\in pvar(A)_{\omega}$. Then $A_1$ is a coordinate algebra over $A$, i.e. there is a system $S$ such that
$$
A_1=\frac{\FX}{\Rad_A(S)}.
$$
Since $\Rad_A(S)=\Rad_B(S)$, so
$$
A_1=\frac{\FX}{\Rad_B(S)},
$$
and hence $A_1$ is a coordinate algebra over $B$. This argument shows that
$$
pvar(A)_{\omega}=pvar(B)_{\omega},
$$
which is a contradiction. Therefore $A$ and $B$ are not geometrically equivalent and this completes the proof of the claim. We can now complete the proof of the theorem. Assume that $C\circ T_A=\Rad_A$. We show that $A$ is geometrically equivalent to any of its filter-powers. So, let $B=A^I/F$ be a filter-power of $A$. Note that we already proved that for a finite system $S_0$, the radicals $\Rad_A(S_0)$ and $\Rad_B(S_0)$ are the same. Suppose that $S$ is an arbitrary system of equations. We have
\begin{eqnarray*}
\Rad_A(S)&=&C(T_A(S))\\
         &=&C(\{ \Rad_A(S_0):\ S_0\subseteq S, |S_0|<\infty\})\\
         &=&C(\{ \Rad_B(S_0):\ S_0\subseteq S, |S_0|<\infty\})\\
         &\subseteq& \Rad_B(S).
\end{eqnarray*}
So we have $\Rad_A(S)=\Rad_B(S)$ and hence $A$ and $B$ are geometrically equivalent. This shows that $A$ is $\qo$-compact. Conversely, let $A$ be $\qo$-compact. For any system $S$, we have
\begin{eqnarray*}
\Rad_A(S)&=&\bigcup_{S_0}\Rad_A(S_0)\\
         &=&\bigvee\{ \Rad_A(S_0): S_0\subseteq S, |S_0|<\infty\}\\
         &=&\bigvee T_A(S),
\end{eqnarray*}
where $\bigvee$ denotes the least upper bound. By our assumption, $C(T_A(S))\subseteq \Rad_A(S)$, so $C(T_A(S))\subseteq \bigvee T_A(S)$. On the other hand, for any finite $S_0\subseteq S$, we have $\Rad_A(S_0)\subseteq C(T_A(S))$. This shows that
$$
C(T_A(S))=\bigvee T_A(S),
$$
and hence $C\circ T_A=\Rad_A$. The proof is now completed.


\begin{thebibliography}{99}

\bibitem{BMR1} Baumslag G., Myasnikov A.,  Remeslennikov V.
{\it Algebraic geometry over groups, I. Algebraic sets and ideal theory}. J.
 Algebra, 1999, {\bf 219}, pp. 16-79.

\bibitem{DMR1}
Daniyarova E., Myasnikov A., Remeslennikov V. {\it Unification theorems in algebraic geometry }. Algebra and Discrete Mathamatics, 2008, {\bf 1}, pp. 80-112.

\bibitem{DMR2}
Daniyarova E., Myasnikov A., Remeslennikov V. {\it Algebraic geometry over algebraic structures, II: Fundations}. J. Math. Sci., 2012, {\bf 185 (3)},
pp. 389-416.

\bibitem{DMR3}
Daniyarova E., Myasnikov A., Remeslennikov V. {\it Algebraic geometry over algebraic structures, III: Equationally noetherian property and
compactness}. South. Asian Bull. Math., 2011, {\bf 35 (1)}, pp. 35-68.

\bibitem{DMR4}
Daniyarova E., Myasnikov A., Remeslennikov V. {\it Algebraic geometry over algebraic structures, IV: Equatinal domains and co-domains}.
Algebra and Logic, {\bf 49 (6)}, pp. 483-508.




\bibitem{ModSH}
Modabberi P., Shahryari M.  {\it Equational conditions in universal algebraic geometry}.
Algebra and Logic, 2015, to appear.


\bibitem{MR}
Myasnikov A., Remeslennikov V. {\it Algebraic geometry over groups, II. Logical Foundations}. J.
 Algebra, 2000, {\bf 234}, pp. 225-276.



\bibitem{Plot}
Plotkin B.
{\it Seven lectures in universal algebraic geometry}, 2002, \verb"arXiv:math/0204245v1".



\end{thebibliography}
\end{document}